\documentclass{amsart}

\usepackage{enumitem}
\usepackage{newclude}
\usepackage{multicol}
\usepackage[labelsep=period]{caption}
\usepackage{tabu}
\usepackage[table]{xcolor}
\usepackage{tikz}
\usetikzlibrary{arrows,automata,positioning}
\usepackage{caption}
\usepackage{rotating}
\usepackage{graphicx}
\usepackage{eurosym}
\usepackage{amsfonts}
\usepackage{amsmath}
\usepackage{amssymb}
\usepackage{stmaryrd}
\usepackage{wasysym}
\usepackage{amsthm}
\usepackage[margin=1in]{geometry}
\usepackage[hang,flushmargin,symbol*]{footmisc}
\usepackage{color}
\definecolor{dark blue}{rgb}{0, 0, .6}
\definecolor{grey}{rgb}{.7, .7, .7}
\definecolor{dark green}{rgb}{0, .4, 0}
\usepackage[breaklinks]{hyperref}

\makeatletter
\newcommand*{\rom}[1]{\expandafter\@slowromancap\romannumeral #1@}
\makeatother

\theoremstyle{definition}
\newtheorem{theorem}{Theorem}[section]

\newtheorem{example}[theorem]{Example}
\newtheorem{conjecture}[theorem]{Conjecture}
\newtheorem{corollary}[theorem]{Corollary}
\newtheorem{lemma}[theorem]{Lemma}
\newtheorem{proposition}[theorem]{Proposition}

\newtheorem{definition}[theorem]{Definition}

\newsavebox{\savepar}

\begin{document}

\title{Perfect Numbers in the Ring of Eisenstein Integers}

\author[Parker]{Zachary Parker}
\author[Rushall]{Jeff Rushall}
\author[Hunt]{Jordan Hunt}
\address{Northern Arizona University, 801 South Osborne Drive, Flagstaff, Arizona 86011}
\email{zdp9@nau.edu, jeffrey.rushall@nau.edu, jdh336@nau.edu}

\begin{abstract}
One of the many number theoretic topics investigated by the ancient Greeks was perfect numbers, which are positive integers equal to the sum of their proper positive integral divisors. Mathematicians from Euclid to Euler investigated these mysterious numbers. We present results on perfect numbers in the ring of Eisenstein integers.

\end{abstract}



\maketitle

%section1
\section{Introduction}\label{sec:intro}

The Pythagoreans, whose motto was \emph{All is Number}, believed that integers possessed mystical and magical properties, which led them to thoroughly investigate the properties of our counting numbers \cite{Burton1991}. For instance, they recognized that 6, 28, 496, and 8128 were the only positive integers less than 10,000 with the property that each was equal to the sum of its proper divisors. This can easily be verified by hand; the proper divisors of $496$ are $1, 2, 4, 8, 16, 31, 62, 124$, and $248$, and of course

\bigskip
\begin{center}
$1 + 2 + 4 + 8 + 16 + 31 + 62 + 124 + 248 = 496$.
\end{center}
\bigskip

Greek mathematicians referred to these special integers as ideal numbers, complete numbers or perfect numbers. Over time, mathematicians chose to use the last of these labels.

\begin{definition}A positive integer $n$ is \textit{perfect} if $n$ is equal to the sum of its proper and positive divisors.
\end{definition}

In \textit{Introduction to Arithmeticae} (circa 100 C.E.), Nicomachus listed the first four perfect numbers, the only ones known to the ancient Greeks. Despite the fact that only a few had been discovered, Euclid, in Book \rom{9}, Proposition $36$ of his classic text \textit{The Elements} (c. 300 B.C.E.) gives sufficient conditions for the existence of even perfect numbers \cite{Burton1991}. More specifically:

\begin{theorem}
Given any integer $k > 1$, if $2^k - 1$ is prime, then $n = 2^{k-1}(2^k - 1)$ is an even perfect number.
\end{theorem}

Leonhard Euler (C.E. 1707--1783) later proved that all even perfect numbers are of Euclid's specified form.  It is worth noting that 33,550,336, the fifth smallest perfect number, remained undiscovered until the 15\textsuperscript{th} century, where it was revealed in an anonymous manuscript. To date, only 49 even perfect numbers have been found.

Integers of the form $2^k - 1$, where $k$ is a positive integer, are known as \emph{Mersenne numbers}, their name derived from Marin Mersenne (C.E. 1588--1648), a monk who had achieved some renown as a number theorist. His investigation of these types of integers was partially inspired by the search for large prime numbers.

\begin{definition} Primes of the form $2^k - 1$, where $k$ is a prime number, are known as \emph{Mersenne primes}. 
\end{definition}

It is important to note that the primality of the exponent in $2^k - 1$ is a necessary but not sufficient condition to guarantee the primality of the corresponding Mersenne number. For instance, the Mersenne number $2^{11} - 1 = 2047$ has a prime exponent, but $2047 = 23 \cdot 89$ is not prime.

As seen in the theorem above and in Euler's subsequent result, Mersenne primes are central to the study of even perfect numbers; both objects will be generalized later in this paper.

In addition to the contributions Euler made concerning even perfect numbers, he proved the following result, which describes the possible structure of an odd perfect number.

\begin{theorem}
If $n$ is an odd perfect number, then $n = p^k \cdot q^2$, where $p\equiv 1 \pmod{4}$ is an odd prime, $k\equiv 1 \pmod{4}$, and the greatest common divisor of $p$ and $q$ is 1.
\end{theorem}

It is unknown if any odd perfect numbers exist; exhaustive computer searches have determined that any odd perfect numbers must exceed $10^{300}$.

The rest of this paper is structured as follows: In Section \ref{sec:back}, we give the appropriate background information from algebraic number theory needed to define the ring of Eisenstein integers, which will serve as the context in which we generalize perfect numbers and Mersenne primes. In Section \ref{sec:mersenne}, we will motivate the existence of Mersenne primes from a slightly different perspective, which will then allow us to rigorously define Eisenstein Mersenne primes. In Section \ref{sec:perfect}, we will define an Eisenstein perfect number and present conditions for the existence of even perfect numbers in this generalized setting. Finally, in Section \ref{sec:work}, we offer concluding remarks and suggest some open questions that require further investigation.

\section{Background}\label{sec:back}

As part of his Ph.D. thesis, the great Karl Friedrich Gauss (C.E. 1777--1855) proved a version of what is now referred to as the \emph{Fundamental Theorem of Algebra}. As a consequence of this result, given any positive integer $n$, a polynomial of the form $x^n - 1$ has exactly $n$ solutions in the complex plane. These solutions are known as \emph{complex} $n\textsuperscript{th}$ \emph{roots of unity}, and it is well-known that they are of the form $e^{2 \pi i k / n}$ for each integer value of $k$, $1 \leq k \leq n$. We may then write

\bigskip
\begin{center}
$x^n - 1 = \prod \limits_{k=1}^n \left( x - e^{2 \pi i k / n} \right)$.
\end{center}
\bigskip

When we restrict this sum to those values of $k$ that are relatively prime to $n$ (i.e., $k$ and $n$ share no common factors), we obtain what is known as the $n\textsuperscript{th}$ \emph{cyclotomic polynomial}, denoted $\Phi_n(x)$. More precisely,

\bigskip
\begin{center}
$\Phi_n(x) = \prod \limits_{(k,n)=1} \left( x - e^{2 \pi i k / n} \right)$.
\end{center}
\bigskip

Cyclotomic polynomials arise in the study of various problems in field theory and algebraic number theory. We are interested exclusively in the third cyclotomic polynomial, namely $\Phi_3(x) = 1 + x + x^2$. The specific root of $\Phi_3(x)$ that will serve as a focus of this paper is this cubic complex root of unity, henceforth denoted $\omega$:

\bigskip
\begin{center}
$e^{2 \pi i /3} = \frac{-1 + \sqrt{-3}}{2}$
\end{center}
\bigskip

This leads to the following proposition.

\begin{proposition}
The set $\mathbb{Z}[\omega] = \{a + b \omega: a,b \in \mathbb{Z} \}$, under the usual operations of addition and multiplication of complex numbers, forms an integral domain, known as the ring of \emph{Eisenstein integers}.
\end{proposition}

It is relatively easy to show that the Eisenstein integers form an integral domain (i.e., a commutative ring with unity having no zero divisors; a proof can be found in many abstract algebra textbooks, such as \cite{Dummit2004} and \cite{Lang2002}). For the remainder of this paper, the term \emph{ring} will always refer to an integral domain. To distinguish between elements of $\mathbb{Z}$ and $\mathbb{Z}[\omega]$, we will often refer to elements (or primes) in $\mathbb{Z}$ as \emph{rational integers} (or \emph{rational primes}).

It will be useful for us to verify the closure axiom under multiplication in $\mathbb{Z}[\omega]$, since doing so will help the reader begin to understand arithmetic in this context. Remember that our complex root of unity $\omega$ satisfies $1 + \omega + \omega ^2 = 0$, and hence $\omega ^2 = -1 - \omega$.

\begin{proposition}
The ring of Eisenstein integers is closed under multiplication.
\end{proposition}

\begin{proof} Given any $a + b \omega, c + d\omega \in \mathbb{Z}[\omega]$, their product is:

\bigskip
\begin{align*}
(a + b \omega) \cdot (c + d \omega) & = ac + (ad) \cdot \omega + (bc) \cdot \omega + (bd) \cdot \omega ^2 \\
& = ac + (ad + bc) \cdot \omega + bd  \cdot (-1 - \omega) \\
& = ac + (ad + bc) \cdot \omega - bd - (bd) \cdot \omega \\
& = (ac - bd) + (ad + bc - bd) \cdot \omega. 
\end{align*}
\bigskip

Since $\mathbb{Z}$ is closed under both addition and multiplication, it then follows that $\mathbb{Z}[\omega]$ is closed under multiplication.
\end{proof}

We will now present several ring-based definitions that are necessary stepping stones toward understanding objects in our particular ring, $\mathbb{Z}[\omega]$. Each are familiar versions of objects and processes in $\mathbb{Z}$.

\begin{definition}
Let $\alpha$ and $\beta$ be elements in a ring $R$. We say that $\alpha$ \emph{divides} $\beta$, written $\alpha | \beta$, if there exists some $\gamma \in R$ such that $\alpha \cdot \gamma = \beta$.
\end{definition}

\begin{definition}
An element $\alpha$ in a ring $R$ is a \emph{unit} if it has a multiplicative inverse -- that is, if there exists some $\beta \in R$ such that $\alpha \cdot \beta = 1$.
\end{definition}

\begin{definition}
An element $\alpha$ of a ring $R$ is \emph{prime} if, whenever $\alpha | \beta\gamma$ for some $\beta, \gamma \in R$, then either $\alpha | \beta$ or $\alpha | \gamma$.
\end{definition}

A visual depiction of $\mathbb{Z}[\omega]$ appears in Figure~\ref{fig:skele}, which shows a selection of Eisenstein integers: The units are pink, the origin dark blue, and random integers light blue. 

\begin{center}
\begin{figure}[h]
\begin{tikzpicture}[scale=2]
\draw [<->] (-3.7,0) --(3.7,0);
\draw [<->] (0,-2.3) --(0,2.3);
[every node/.style=draw, every label/.style=draw]

\draw [ultra thin] (1,0) -- ++(0:2) -- ++(60:1) -- ++(120:1) -- ++(180:6) -- ++(-120:1) -- ++(-60:1) -- ++(-120:1) -- ++(-60:1) -- ++(0:6) -- ++(60:1) -- ++(120:3) -- ++(-120:4) -- ++(120:4) -- ++(-120:2) -- ++(-60:2) -- ++(60:4) -- ++(-60:4) -- ++(60:3) -- ++(180:6) -- ++(-60:3) -- ++(60:4) -- ++(-60:3) -- ++(180:5) -- ++(60:3) -- ++(-60:4) -- ++(0:2) -- ++(60:1) -- ++(180:7) -- ++(-60:1) -- ++(60:2) -- ++(120:2) -- ++(-120:1) -- ++(0:1) -- ++(-60:3) -- ++(0:2) -- ++(60:4) -- ++(-60:1) -- ++(-120:2) -- ++(-60:1);
\draw [ultra thin] (-3,0) -- (1,0);
\filldraw (1,0)  ++(0:1) circle(1pt) ++(120:1) circle(1pt) ++(0:1) circle(1pt) ++(-120:1) circle(1pt) ++(0:1) circle(1pt) ++(120:2) circle(1pt) ++(-120:1) circle(1pt) ++(-120:1) circle(1pt) ++(-120:1) circle(1pt) ++(0:1) circle(1pt) ++(120:1) circle(1pt) ++(0:1) circle(1pt) ++ (-120:1) circle(1pt) ++(0:1) circle(1pt) ++(120:1) circle(1pt) ++(0:1) circle(1pt) ++(-120:1) circle(1pt) ++(0:1) circle(1pt) ++(120:1) circle(1pt) ++(60:1) circle(1pt) ++(180:1) circle(1pt) ++(60:1) circle(1pt) ++(180:1) circle(1pt) ++(180:1) circle(1pt) ++(-120:1) circle(1pt) ++(0:1) circle(1pt) ++(120:1) circle(1pt) ++(180:1) circle(1pt) ++(-60:1) circle(1pt) ++(-60:1) circle(1pt) ++(180:1) circle(1pt) ++(60:1) circle(1pt) ++(180:1) circle(1pt) ++(60:1) circle(1pt) ++(180:1) circle(1pt) ++(-60:1) circle(1pt) ++(-60:1) circle(1pt) ++(-60:1) circle(1pt) ++(180:1) circle(1pt) ++(60:1) circle(1pt) ++(120:1) circle(1pt) ++(-120:1) circle(1pt) ++(-60:1) circle(1pt) ++(-60:1) circle(1pt) ++(60:1) circle(1pt) ++(-60:1) circle(1pt) ++(60:1) circle(1pt) ++(-60:1) circle(1pt) ++(60:1) circle(1pt) ++(-60:1) circle(1pt) ++(60:1) circle(1pt) ++(-120:1) circle(1pt) ++(180:6) circle(1pt) ++(60:1) circle(1pt) ++(-60:1) circle(1pt) ++(60:1) circle(1pt) ++(-60:1) circle(1pt) ++(60:1) circle(1pt) ++(180:2) circle(1pt) ++(120:1) circle(1pt) ++(0:1) circle(1pt) ++(-120:1) circle(1pt) ++(0:1) circle(1pt) ++(60:1) circle(1pt) ++(180:1) circle(1pt) ++(-60:1) circle(1pt) ++(60:3) circle(1pt) ++(180:3) circle(1pt) ++(-60:1) circle(1pt) ++(60:1) circle(1pt) ++(-60:1) circle(1pt) ++(60:1) circle(1pt) ++(-60:1) circle(1pt) ++(180:3) circle(1pt) ++(-60:1) circle(1pt) ++(60:1) circle(1pt) ++(-60:1) circle(1pt) ++(60:1) circle(1pt) ++(-60:1) circle(1pt) ++(180:2) circle(1pt) ++(-120:1) circle(1pt) ++(0:1) circle(1pt) ++(-120:1) circle(1pt) ++(120:1) circle(1pt) ++(60:3) circle(1pt) ++(180:1) circle(1pt) ++(-120:1) circle(1pt) ++(0:7) circle(1pt) ++(120:1);

\color{magenta}
\node [label={[label distance=1.9cm]58:$1 + \omega$}]{};
\node [label={[label distance=1.3cm]313:$- \omega$}]{};
\node[above right] at (1,0) {$1$};
\node[above right] at (-1,0) {$-1$};
\node [label={[label distance=1.65cm]99:$\omega$}]{};
\node [label={[label distance=.88cm]267:$-1 - \omega$}]{};
\color{black}
\node[above right] at (2,0) {$2$};
\node[above right] at (3,0) {$3$};
\node[above right] at (-2,0) {$-2$};
\node[above right] at (-3,0) {$-3$};
\color{dark blue}
\node[above left] at (0,0) {$O$};
\color{cyan}
\node [label={[label distance=3.88cm]59:$2 + 2 \omega$}]{};
\node [label={[label distance=2.98cm]301:$-2 \omega$}]{};
\node [label={[label distance=4.9cm]347:$2 - \omega$}]{};
\node [label={[label distance=6.7cm]15:$4 + \omega$}]{};
\node [label={[label distance=5.6cm]143:$-2 + 2 \omega$}]{};
\node [label={[label distance=3.9cm]195:$-3 - \omega$}]{};

\end{tikzpicture}
\caption{}
\label{fig:skele}
\end{figure}
\end{center}

An important tool used to study $\mathbb{Z}[\omega]$ is the norm function. Informally, one may think of the norm function as a kind of ``measuring the distance from the origin" device, not unlike the role the absolute value function plays over $\mathbb{R}$.

\begin{definition}
Given any $\alpha = a + b \omega$ in $\mathbb{Z}[\omega]$, the \emph{norm function} $N:\mathbb{Z}[\omega]\rightarrow \mathbb{Z}$ is defined as

\bigskip
\begin{center}
$N(\alpha) = a^2 - ab + b^2$.
\end{center}
\end{definition}
\bigskip

The norm function possesses some useful qualities. First, the function always provides a nonnegative integer output, and in fact takes on nonzero values at all nonzero Eisenstein integers. We leave it to the curious reader to verify this fact. Second, the norm function is completely multiplicative, as shown in the following lemma.

\begin{lemma}
If $\alpha, \beta \in \mathbb{Z}[\omega]$, then $N(\alpha \cdot \beta) = N(\alpha) \cdot N(\beta)$.
\end{lemma}

\begin{proof}
Let $\alpha = a + b \omega$ and $\beta = c + d \omega$. Then:

\bigskip
\begin{align*}
N(\alpha \cdot \beta) & = N\left[ (a + b \omega) \cdot (c + d \omega) \right]\\
& = N \left[ (ac - bd) + (ad + bc - bd) \cdot \omega \right]\\
& = \left[ ac - bd \right] ^2 - \left[ ac - bd \right] \left[ ad + bc - bd \right] + \left[ad + bc - bd \right] ^2\\
& = a^2 c^2 -2abcd +b^2 d^2 - \left[ a^2 cd + abc^2 - abcd - abd^2 - b^2 cd + b^2 d^2 \right] \\
& + \left[a^2 d^2 + abcd - abd^2 + abcd + b^2 c^2 - b^2 cd - abd^2 - b^2cd +b^2 d^2 \right] \\
& = a^2c^2 - a^2cd + a^2d^2- abc^2 + abcd - abd^2 + b^2c^2 - b^2cd + b^2d^2 \\
& = a^2(c^2 - cd + d^2) - ab(c^2 - cd + d^2) + b^2(c^2 - cd + d^2) \\
& = (a^2 - ab + b^2)(c^2 - cd + d^2) \\
& = N(\alpha)N(\beta).
\end{align*}
\end{proof}

Using the norm function, it can be shown that $\mathbb{Z}[\omega]$ is a Euclidean ring. A key step in this proof is verifying that the norm function imposes enough structure on $\mathbb{Z}[\omega]$ to allow for something akin to the division algorithm to exist. Specifically, one needs the following result, whose proof can be found in many sources, such as \cite{Cox2013}.

\begin{proposition}
If $\alpha, \beta \in \mathbb{Z}[\omega], \beta \neq 0$, there exist $\gamma, \delta \in \mathbb{Z}[\omega]$ that satisfy $\alpha = \gamma \beta + \delta$, and such that $N(\delta) < N(\beta)$.
\end{proposition}

Since $\mathbb{Z}[\omega]$ is a Euclidean ring, it is also a principal ideal domain and hence a unique factorization domain, a chain of implications that is discussed in most abstract algebra textbooks (see, for instance, \cite{Lang2002}). Naturally, we must now determine those elements of $\mathbb{Z}[\omega]$ that are units; we leave verification of the next result to the curious reader.

\begin{proposition}
An element $\alpha$ of $\mathbb{Z}[\omega]$ is a unit if and only if $N(\alpha) = 1$.
\end{proposition}

It then follows that the only units of $\mathbb{Z}[\omega]$ are $\{\pm 1, \pm \omega, \pm \omega ^2\}$. In any given ring, elements that differ only via multiplication by a unit are known as $associates$. Thus, every element in $\mathbb{Z}[\omega]$ has six different associates. Later in this section, we will show how to select the appropriate associate to allow our generalizations in $\mathbb{Z}[\omega]$ of perfect numbers to be well-defined.
This complete classification of units in $\mathbb{Z}[\omega]$ can be used to help determine when a given element in the ring of Eisenstein integers is prime.

\begin{proposition}
Let $\alpha \in \mathbb{Z}[\omega]$. If $N(\alpha)$ is a rational prime in the integers, then $\alpha$ is prime in the Eisenstein integers.
\end{proposition}

\begin{proof}
Let $\alpha \in \mathbb{Z}[\omega]$. Assume $N(\alpha)$ is prime and suppose $\alpha = \beta \cdot \gamma$ for some $\beta, \gamma \in \mathbb{Z}[\omega]$. Then:

\bigskip
\begin{align*}
N(\alpha) & = N(\beta \cdot \gamma)\\
& = N(\beta) \cdot N(\gamma).
\end{align*}
\bigskip

Because $N(\alpha)$ is by assumption an integer prime, we may assume without loss of generality that $N(\beta)=1$. Thus, $\beta$ is a unit, which implies that $\alpha$ and $\gamma$ are associates, and the result follows.
\end{proof}

It follows directly from the definitions and prior results that if $\alpha$ and $\gamma$ are associates in $\mathbb{Z}[\omega]$, then $N(\alpha) = N(\gamma)$. This is another property of the norm function that mimics the absolute value function over $\mathbb{R}$, and will be important in later sections.

\begin{example} The Eisenstein integer $1 - \omega$ is a prime in $\mathbb{Z}[\omega]$. This follows by simply applying our previous proposition:
\begin{align*}
N(1 - \omega) & = 1^2 - 1 \cdot (-1) + (-1)^2 \\
& = 1 + 1 + 1 \\
& = 3,
\end{align*}
\bigskip

which is prime in $\mathbb{Z}$. We leave it to the curious reader to verify that, up to multiplication by a unit, $1 - \omega$ is the prime of minimal norm value in $\mathbb{Z}[\omega]$.
\end{example}

The following lemma completely determines which integer primes remain prime in $\mathbb{Z}[\omega]$; a formal proof may be found in \cite{Cox2013}.

\begin{proposition}
Let $p$ be prime in $\mathbb{Z}$.
\begin{itemize}
\item If $p = 3$, then $p$ is not prime in $\mathbb{Z}[\omega]$.
\item If $p \equiv 1 \pmod{3}$, there exists a prime $\pi \in \mathbb{Z}[\omega]$ such that $p = \pi \overline \pi$, where $\overline \pi$ denotes the complex conjugate of $\pi$.
\item If $p \equiv 2 \pmod{3}$, then $p$ remains prime in $\mathbb{Z}[\omega]$.
\end{itemize}
\end{proposition}

This proposition highlights some of the fundamental computational differences between $\mathbb{Z}$ and $\mathbb{Z}[\omega]$. For instance, observe that $3 = -\omega ^2 \cdot (1 - \omega) ^2$ (the reader may easily verify this statement). Consequently, the rational prime $3$, having a nontrivial factorization in the ring of Eisenstein integers, is no longer prime in $\mathbb{Z}[\omega]$. Similarly, note that $7 \equiv 1 \pmod{3}$, and that $7 = -\omega ^2 \cdot (2 - \omega) ^2$. Thus, the rational prime $7$ is not a prime in $\mathbb{Z}[\omega]$.

Another useful and necessary generalization from $\mathbb{Z}$ to $\mathbb{Z}[\omega]$ is that of evenness. As with the rational integers, we classify evenness in the Eisenstein integers via divisibility by a prime, specifically $1 - \omega$.

\begin{definition} 
An element $a + b\omega \in \mathbb{Z}[\omega]$ is \emph{even} if it is divisible by $1 - \omega$ - that is, if there exists some $c + d\omega \in \mathbb{Z}[\omega]$ such that $a + b\omega = (1 - \omega)(c + d\omega)$. Any element of $\mathbb{Z}[\omega]$ that is not divisible by $1 - \omega$ is said to be \emph{odd}.
\end{definition}

The reason for the reliance on $1 - \omega$ for determining evenness in $\mathbb{Z}[\omega]$ will be made clear in the next section. The following contains an easy computational test for verifying evenness in the Eisenstein sense.

\begin{theorem}
An Eisenstein integer $a + b \omega$ is even if and only if $a + b \equiv 0 \pmod{3}$.
\end{theorem}

\begin{proof}
Suppose $a + b \omega$ is even. We may then write $a + b\omega = (1 - \omega)(c + d\omega)$ for some $c + d \omega \in \mathbb{Z}[\omega]$. Simplifying yields
\begin{align*}
(a + b \omega) & = (1 - \omega) \cdot (c + d \omega)\\
& = c + d \omega - c \omega - d \omega ^2\\
& = c + d \omega - c \omega + d +d \omega\\
& = (c + d) + (2d - c)\omega.
\end{align*}
\bigskip

We must now consider the nine possible pairs of congruence classes modulo 3 to which $c$ and $d$ belong:

\begin{itemize}
\item If $c \equiv d \equiv 0 \pmod{3}$, it then easily follows that $a \equiv b \equiv 0 \pmod{3}$, and hence $a + b \equiv 0 \pmod{3}$.

\item If $c \equiv 0 \pmod{3}$ and $d \equiv 1 \pmod{3}$, it again follows that $a \equiv 1 \pmod{3}$ and $b \equiv 2 \pmod{3}$, and hence $a + b \equiv 0 \pmod{3}$.

\item If $c \equiv 0 \pmod{3}$ and $d \equiv 2 \pmod{3}$, it again follows that $a \equiv 2 \pmod{3}$ and $b \equiv 1 \pmod{3}$, and hence $a + b \equiv 0 \pmod{3}$.
\end{itemize}
\par
The remaining cases are similar and left to the reader. Thus, if $a + b \omega$ is even, then $a + b \equiv 0 \pmod{3}$. For the reverse direction, suppose that $a + b \equiv 0 \pmod{3}$. As above, we consider the nine possible pairs of congruence classes modulo 3 to which $a$ and $b$ belong. 
\begin{itemize}
\item If $a \equiv b \equiv 0 \pmod{3}$, then $a + b\omega = 3(c + d\omega)$ for some $c + d \omega \in \mathbb{Z}[\omega]$. But 3 is not prime in $\mathbb{Z}[\omega]$, and in fact 3 is divisible by $1 - \omega$. Thus, $a + b \omega$ is divisible by $1 - \omega$ as well, and hence $a + b \omega$ is even.
\item If $a \equiv 1 \pmod{3}$ and $b \equiv 2 \pmod{3}$, then $a = 3k + 1$ and $b = 3l + 2$ for some $k,l \in \mathbb{Z}$. Then $a + b\omega = (3k + 1) + (3l +2)\omega = 3(k + l\omega) + (1 + 2\omega)$. But both 3 and $1 + 2\omega$ are divisible by $1 - \omega$, the latter due to the fact that $(1 - \omega) \cdot \omega = 2 + \omega$. Again, we may conclude that $a + b\omega$ is even.
\end{itemize}

The remaining cases are similar in nature and left to the reader to verify.  Thus, if $a + b \equiv 0 \pmod{3}$, then $a + b\omega$ is even, and the result follows.
\end{proof}

Another useful property of the norm function is that it provides a different test for evenness in $\mathbb{Z}[\omega]$. This test is based on the fact that, given any $a + b\omega \in \mathbb{Z}[\omega]$, when considering the nine possible pairs of congruence classes modulo 3 to which $a$ and $b$ can belong, it is simple to show that either $N(a+b\omega)\equiv 0 \pmod{3}$ or $N(a+b\omega)\equiv 1 \pmod{3}$. This leads directly to the following result.

\begin{lemma}
An Eisenstein integer $a+b\omega$ is even if and only if $N(a+b\omega)\equiv 0 \pmod{3}$.
\end{lemma}

\begin{proof}
Assume $a+b\omega$ is even -- say, $a+b\omega = (1-\omega)(c+d\omega)$ for some $c+d \omega$. Then by our previous work, we have that
\begin{align*}
N(a+b\omega) & = N((1-\omega)(c+d\omega))\\
& = N(1-\omega)N(c+d\omega)\\
& =(1^2-(1)(-1)+(-1)^2)(N(c+d\omega))\\
& = 3 \cdot N(c + d\omega)\\
& \equiv 0 \pmod{3}.
\end{align*}
\bigskip

On the other hand, if $N(a + b \omega) \equiv 0 \pmod{3}$, then a simple case-by-case check, as seen in the proof of Theorem 2.14, shows that all such elements $a+b\omega$ are divisible by $1 - \omega$, and the result follows.
\end{proof}

This result, together with the fact that the norm function $N$ only achieves integer values that are congruent to $0$ or $1$ modulo 3, leads directly to the following corollary.

\begin{corollary}
An Eisenstein integer $a+b\omega$ is not even if and only if $N(a+b\omega)\equiv 1 \pmod{3}$.
\end{corollary}

%section3
\section{The $\sigma$ Function and Mersenne Primes in $\mathbb{Z}[\omega]$}\label{sec:mersenne}

An important number theoretic tool that can be extended from  $\mathbb{Z}$ to $\mathbb{Z}[\omega]$ and will be useful for our purposes is the classic sum of divisors function.

\begin{definition}
Given any $n \in \mathbb{Z}$, the \emph{sum of divisors function} $\sigma(n)$ denotes the sum of positive factors of $n$. That is,
\begin{center}
$\sigma(n) = \sum \limits_{d | n} d$.
\end{center}
\end{definition}

It should be noted that perfect numbers are positive integer solutions to the equation $\sigma(n) = 2n$. The sum of divisors function has several useful and well-known properties, including those in the next proposition, whose proof is left to the reader.

\begin{proposition}
Let $p$ be a prime and $n$ be any integer. Then:
\begin{itemize}
\item $\sigma(p) = p+1$
\item $\sigma(p^{n}) = \frac{p^{n+1}-1}{p-1}$
\item If $m$ and $n$ are relatively prime, then $\sigma(m \cdot n) = \sigma(m) \cdot \sigma(n)$.
\end{itemize}
\end{proposition}

In particular, note that $\sigma (n)$ is multiplicative. A direct consequence of these properties is that we can compute the value of $\sigma (n)$ at every positive integer $n$. That is, writing $n = p_1^{\alpha _1}  p_2^{\alpha _2}  \cdots  p_s^{\alpha _s}$ where $p_1, \ldots, p_s$ are distinct primes and $\alpha _1, \ldots, \alpha _s$ are positive integers, then

\bigskip
\begin{align*}
\sigma(n) & = \sigma(p_1^{\alpha _1}) \cdot \sigma(p_2^{\alpha _2}) \cdots \sigma(p_s^{\alpha _s}) \\
& = \frac{p_1 ^{\alpha _1 +1} -1}{p_1 -1} \cdot \frac{p_2 ^{\alpha _2 +1} -1}{p_2 -1} \cdots \frac{p_s ^{\alpha _s +1} -1}{p_s -1} \\
& = \prod \limits_{j=1}^{s} \frac{p_j ^{\alpha _j +1} -1}{p _j - 1}.
\end{align*}
\bigskip

To extend $\sigma$ to a well-defined function $\sigma^\star$ on $\mathbb{Z}[\omega]$, we take advantage of that fact that the ring of Eisenstein integers features unique factorization, together with the geometric properties inherent in the multiplication of complex numbers. More precisely, recall that every Eisenstein integer $\alpha$ has six associates, since there are six units in $\mathbb{Z}[\omega]$. Initially, one has the freedom to select any of these six associates to serve as the representative when computing something akin to a sum of divisors in $\mathbb{Z}[\omega]$. However, the $\sigma$--function features the property that $\sigma(n) \geq n$ for every $n \in \mathbb{N}$. Thus, it would be advantageous to pick a particular associate $\alpha '$ of $\alpha$ that satisfies

\begin{center}
$N(\sigma^{\star}(\alpha ')) \geq N(\alpha ')$.
\end{center}
\bigskip

The associates that satisfy this inequality are precisely those located in our shaded region of the complex plane, as indicated in Figure~\ref{fig:tri}. We may assume without loss of generality that the shaded region in Figure~\ref{fig:tri} contains the Eisenstein integers located on the lower boundary (and not the upper boundary). Note that this shaded region contains the Eisenstein integer $a + b\omega$ if and only if $a > b \geq 0$.

\begin{center}
\begin{figure}[h]
\begin{tikzpicture}[scale=2]
[every node/.style=draw, every label/.style=draw]
\color{grey}
\draw [ultra thin] (1,0) -- ++(0:2) -- ++(60:1) -- ++(120:1) -- ++(180:6) -- ++(-120:1) -- ++(-60:1) -- ++(-120:1) -- ++(-60:1) -- ++(0:6) -- ++(60:1) -- ++(120:3) -- ++(-120:4) -- ++(120:4) -- ++(-120:2) -- ++(-60:2) -- ++(60:4) -- ++(-60:4) -- ++(60:3) -- ++(180:6) -- ++(-60:3) -- ++(60:4) -- ++(-60:3) -- ++(180:5) -- ++(60:3) -- ++(-60:4) -- ++(0:2) -- ++(60:1) -- ++(180:7) -- ++(-60:1) -- ++(60:2) -- ++(120:2) -- ++(-120:1) -- ++(0:1) -- ++(-60:3) -- ++(0:2) -- ++(60:4) -- ++(-60:1) -- ++(-120:2) -- ++(-60:1);
\draw [ultra thin] (-3,0) -- (3,0);
\draw [ultra thin] (0,-1.7) -- (0,1.7);
\filldraw (1,0)  ++(0:1) circle(1pt) ++(120:1) circle(1pt) ++(0:1) circle(1pt) ++(-120:1) circle(1pt) ++(0:1) circle(1pt) ++(120:2) circle(1pt) ++(-120:1) circle(1pt) ++(-120:1) circle(1pt) ++(-120:1) circle(1pt) ++(0:1) circle(1pt) ++(120:1) circle(1pt) ++(0:1) circle(1pt) ++ (-120:1) circle(1pt) ++(0:1) circle(1pt) ++(120:1) circle(1pt) ++(0:1) circle(1pt) ++(-120:1) circle(1pt) ++(0:1) circle(1pt) ++(120:1) circle(1pt) ++(60:1) circle(1pt) ++(180:1) circle(1pt) ++(60:1) circle(1pt) ++(180:1) circle(1pt) ++(180:1) circle(1pt) ++(-120:1) circle(1pt) ++(0:1) circle(1pt) ++(120:1) circle(1pt) ++(180:1) circle(1pt) ++(-60:1) circle(1pt) ++(-60:1) circle(1pt) ++(180:1) circle(1pt) ++(60:1) circle(1pt) ++(180:1) circle(1pt) ++(60:1) circle(1pt) ++(180:1) circle(1pt) ++(-60:1) circle(1pt) ++(-60:1) circle(1pt) ++(-60:1) circle(1pt) ++(180:1) circle(1pt) ++(60:1) circle(1pt) ++(120:1) circle(1pt) ++(-120:1) circle(1pt) ++(-60:1) circle(1pt) ++(-60:1) circle(1pt) ++(60:1) circle(1pt) ++(-60:1) circle(1pt) ++(60:1) circle(1pt) ++(-60:1) circle(1pt) ++(60:1) circle(1pt) ++(-60:1) circle(1pt) ++(60:1) circle(1pt) ++(-120:1) circle(1pt) ++(180:6) circle(1pt) ++(60:1) circle(1pt) ++(-60:1) circle(1pt) ++(60:1) circle(1pt) ++(-60:1) circle(1pt) ++(60:1) circle(1pt) ++(180:2) circle(1pt) ++(120:1) circle(1pt) ++(0:1) circle(1pt) ++(-120:1) circle(1pt) ++(0:1) circle(1pt) ++(60:1) circle(1pt) ++(180:1) circle(1pt) ++(-60:1) circle(1pt) ++(60:3) circle(1pt) ++(180:3) circle(1pt) ++(-60:1) circle(1pt) ++(60:1) circle(1pt) ++(-60:1) circle(1pt) ++(60:1) circle(1pt) ++(-60:1) circle(1pt) ++(180:3) circle(1pt) ++(-60:1) circle(1pt) ++(60:1) circle(1pt) ++(-60:1) circle(1pt) ++(60:1) circle(1pt) ++(-60:1) circle(1pt) ++(180:2) circle(1pt) ++(-120:1) circle(1pt) ++(0:1) circle(1pt) ++(-120:1) circle(1pt) ++(120:1) circle(1pt) ++(60:3) circle(1pt) ++(180:1) circle(1pt) ++(-120:1) circle(1pt) ++(0:7) circle(1pt) ++(120:1);

\coordinate (A1) at (0,0);
\coordinate (A2) at ++(60:2.1);
\coordinate (A3) at ++(0:4.1);
\coordinate (A4) at ++(30.5:3.56);

\draw[dashed, ultra thick, black] (A1) -- (A2);
\draw[ultra thick, black] (A1) -- (A3);

\shade [left color=black,right color=grey, opacity=0.4] (A1) -- (A3) -- (A4)  -- (A2) -- cycle;

\color{magenta}
\node [label={[label distance=1.9cm]58:$1 + \omega$}]{};
\node [label={[label distance=1.3cm]313:$- \omega$}]{};
\node[above right] at (1,0) {$1$};
\node[above right] at (-1,0) {$-1$};
\node [label={[label distance=1.65cm]99:$\omega$}]{};
\node [label={[label distance=.88cm]267:$-1 - \omega$}]{};
\color{dark blue}
\node[above left] at (0,0) {$O$};

\end{tikzpicture}
\caption{}
\label{fig:tri}
\end{figure}
\end{center}

Using these particular associates in $\mathbb{Z}[\omega]$ allows us to generalize $\sigma (n)$ in a well-defined manner. Moreover, selecting these particular prime associates in $\mathbb{Z}[\omega]$ allows, among other things, for $\mathbb{Z}[\omega]$ to contain all primes in $\mathbb{Z}$ that remain prime in $\mathbb{Z}[\omega]$. Consequently, we may think of $\sigma^\star$ as being an extension of $\sigma$ from $\mathbb{Z}$ to $\mathbb{Z}[\omega]$.

\newpage

Let $\nu \in \mathbb{Z}[\omega]$ be arbitrary. Via unique factorization in $\mathbb{Z}[\omega]$ and without loss of generality, $\nu$ can be written as

\begin{center}
$\nu = \epsilon \prod \limits_{q=1}^{s}{\pi _q ^{k_q}}$,
\end{center}

\bigskip

\begin{flushleft}
where $\epsilon$ is a unit, each $\pi _q$ is a prime in the shaded section of Figure~\ref{fig:tri}, and each $k_q$ is a positive integer. For the rest of this paper, we will assume that all prime factorizations of elements in $\mathbb{Z}[\omega]$ are of this form. We can now define the complex sum of divisors function $\sigma ^ \star (\nu)$ in the following manner.
\end{flushleft}

\begin{definition}
Given any $\nu = \epsilon \prod \limits_{q=1}^{s}{\pi _q ^{k_q}}$ in $\mathbb{Z}[\omega]$, the expression $\sigma ^\star (\nu)$, the \emph{complex sum of divisors function} evaluated at $\nu$, is defined as
\begin{center}
$\sigma ^\star (\nu) = \prod \limits_{q=1}^{s} \frac{\pi _q ^{ k_q +1} -1}{\pi _q - 1}$.
\end{center}
\end{definition}
\bigskip

Note that if $\nu \in \mathbb{Z}[\omega]$ has no imaginary part, then $\nu \in \mathbb{Z}$, in which case the definition of $\sigma ^\star$ is equivalent to that of $\sigma$ in $\mathbb{Z}$. Moreover, $\sigma ^\star$ is multiplicative over $\mathbb{Z}[\omega]$, a property that will play a key role when attempting to define perfect numbers in $\mathbb{Z}[\omega]$.

We will now generalize the concepts of Mersenne numbers and Mersenne primes from the rational integers to $\mathbb{Z}[\omega]$. To begin, we first give a somewhat different motivation for the structure of Mersenne primes.

Consider the expression $b^n - 1$, where $n$ is any positive integer and, for the moment, assume that $b \in \mathbb{Z}$. Because $b^n - 1 = (b - 1)(b^{n-1} + b^{n-2} + \cdots + b + 1)$, in order for $b^n - 1$ to be prime in $\mathbb{Z}$, exactly one of $b - 1$ and $b^{n-1} + b^{n-2} + \cdots + b + 1$ must be a unit in $\mathbb{Z}$. As Pershell and Huff demonstrate in \cite{Pershell2002}, this is only possible when $b - 1$ is a unit -- that is, when $b - 1 = \pm 1$, in which case $b = 0$ or $b = 2$. Clearly, $b = 0$ yields no prime, and so only the $b = 2$ case is worth pursuing.

Now consider the expression $b^n - 1$, but assume that $b \in \mathbb{Z}[\omega]$. Once again, in order for $b^n - 1$ to be prime in $\mathbb{Z}[\omega]$, exactly one of $b - 1$ and $b^{n-1} + b^{n-2} + \cdots + b + 1$ must be one of the six units in $\mathbb{Z}[\omega]$. There are four possibilities:

\begin{itemize}
\item If $b - 1 = \pm 1$, we again find that $b = 2$ or $b = 0$. The former yields rational Mersenne primes, while the latter yields no primes.
\item If $b - 1 = \omega$, then $b = 1 + \omega = -\omega^2$, which is a unit, and the corresponding expression $b^n - 1$ yields only the primes $\omega - 1$ and $\omega^2 - 1$.
\item If $b - 1 = \omega^2$, then $b = 1 + \omega^2 = -\omega$, which is a unit, and again the corresponding expression $b^n - 1$ yields only two primes, which are complex conjugates of the last case.
\item Finally, if $b - 1 = - \omega$, then $b = 1 - \omega$ (the case where $b - 1 = - \omega^2$ differs only by multiplication by a unit, and yields no new information). This particular case requires further investigation.
\end{itemize}

As demonstrated in \cite{Pershell2002}, when considering the possibility of $b^{n-1} + b^{n-2} + \cdots + b + 1$ being a unit, no new information is gleaned. Thus, the expression $b^n - 1$ can only be prime in $\mathbb{Z}[\omega]$ when $b = 1 - \omega$. The curious reader can refer to the aforementioned work by Pershell and Huff for more information.

It is for these reasons that the prime $2$ plays a key role when defining a Mersenne prime in $\mathbb{Z}$, and why $1 - \omega$ should play the same role when attempting to generalize Mersenne primes to $\mathbb{Z}[\omega]$. Moreover, since the prime of smallest norm in $\mathbb{Z}[\omega]$ is $1 - \omega$, it seems reasonable to use $1 - \omega$ in this context. Consequently, the following definitions should come as no surprise.

\begin{definition}
Given any positive rational integer $k$, an \emph{Eisenstein Mersenne number}, denoted $M_k$, is an element of $\mathbb{Z}[\omega]$ of the form $(1 - \omega)^k - 1$, and an \emph{Eisenstein Mersenne prime} is a prime in $\mathbb{Z}[\omega]$ of the form $(1 - \omega)^p - 1$ where $p$ is a prime in $\mathbb{Z}$.
\end{definition}

Below we provide examples of both prime and non-prime Eisenstein Mersenne numbers. Note the role that the norm function plays in these verifications.

\begin{example}
Let $k = 2$ (which of course is prime in $\mathbb{Z}$) in the definition of an Eisenstein Mersenne number. The resulting expression simplifies to
\begin{align*}
M_2 & = (1 - \omega)^2 - 1 \\& = (1 - 2 \omega + \omega ^2) - 1\\
& = -2 \omega + \omega ^2\\
& = -2 \omega - 1 - \omega\\
& = -1 - 3 \omega.
\end{align*}
\bigskip

We now compute the norm of this simplified Eisenstein integer:

\bigskip
\begin{align*}
N(-1 - 3 \omega) & = (-1)^2 - (-1)(-3) + (-3)^2\\
& = 1 - 3 +9\\
& = 7.
\end{align*}
Because $7$ is prime in $\mathbb{Z}$, we know via Lemma 2.7 that $M_2 = (1 - \omega)^2 - 1$ is prime in $\mathbb{Z}[\omega]$ and hence is an Eisenstein Mersenne prime.
\end{example}

\begin{example}
Let $k = 3$ in the definition of an Eisenstein Mersenne number. Then

\bigskip
\begin{align*}
M_3 & = (1 - \omega)^3 - 1\\
& = (1 - 3 \omega + 3 \omega ^2 - \omega ^3) -1\\
& = -3 \omega + 3 \omega ^2 - \omega ^3\\
& = -3 \omega - 3 - 3 \omega - 1\\
& = -4 - 6 \omega.
\end{align*}
\bigskip

Since both coefficients are even rational integers, it is clear that this Eisenstein integer can be factored in a nontrivial way in $\mathbb{Z}[\omega]$. But the norm function also verifies this:

\bigskip
\begin{align*}
N(-4 - 6 \omega) & = (-4)^2 - (-4)(-6) + (-6)^2\\
& = 16 - 24 +36\\
& = 28.
\end{align*}
\bigskip
As a result, the Eisenstein Mersenne number $M_3 = (1 - \omega)^3 - 1$ is not prime in $\mathbb{Z}[\omega]$.
\end{example}

Eisenstein Mersenne primes are in some sense even rarer than Mersenne primes in $\mathbb{Z}$. For instance, there are 12 rational primes of the form $2^n - 1$ with $n \leq 160$, but only 7 Eisenstein Mersenne primes exist in the same exponent range, namely $n = 2, 5, 7, 11, 17, 19$ and $79$. The curious reader may consult \cite{Pershell2002} for more information.

Observe that the exponent in the expression $(1 - \omega)^k - 1$  is necessarily prime in $\mathbb{Z}$; if it were not, $(1 - \omega)^k - 1$ would have, as is the case with Mersenne numbers in $\mathbb{Z}$, at least one nontrivial factor. It is also important to reiterate that having a prime exponent in the expression $(1 - \omega)^k - 1$ is a necessary but not sufficient condition to guarantee the primality of the corresponding Eisenstein Mersenne number.

%section4
\section{Perfect Numbers in $\mathbb{Z}[\omega]$}\label{sec:perfect}

Recall that a perfect number is a positive integer $n$ such that $\sigma(n) = 2n$. The use of our complex sum-of-divisors function $\sigma^\star$ allows us to generalize perfect numbers in a seemingly natural way from $\mathbb{Z}$ to $\mathbb{Z}[\omega]$.

\begin{definition}
An Eisenstein integer $\alpha$ is said to be \emph{perfect} if $\sigma ^ \star (\alpha) = (1- \omega) \cdot \alpha$.
\end{definition}

This characterization of perfectness in the Eisenstein integers seems both logical and appropriate. But as others, such as Spira in \cite{Spira1961} and McDaniel in \cite{McDaniel1974} have noted, this particular generalization of perfect numbers, while natural to consider, is somewhat limited. Using a norm function, both Spira and McDaniel gave an alternative definition of perfect numbers in $\mathbb{Z}[i]$, the ring of Gaussian integers. We follow suit, noting that $1 - \omega$ is the prime of minimal norm in $\mathbb{Z}[\omega]$, and in fact $N(1 - \omega) = 3$.

\begin{definition}
An Eisenstein integer $\alpha$ is said to be \emph{norm--perfect} if $N[\sigma^\star(\alpha)] = 3 \cdot N[\alpha]$.
\end{definition}

A moment's thought reveals that every perfect Eisenstein integer is in fact a norm--perfect Eisenstein integer. Regardless of which definition we choose to investigate, one key point must be addressed. Namely, in Euclid's proof of the structure of even perfect numbers in $\mathbb{Z}$, we find the following:

\bigskip
\begin{center}
$\sigma(2^k - 1) = 2^k - 1 + 1 = 2^k$ 
\end{center}
\bigskip

This statement is valid, of course, because the expression $2^k -1$ is a Mersenne prime. However, in our generalized context, the complex sum of divisors function $\sigma^\star$ requires inputs that are specific associates, namely, those that live in our region as depicted in Figure~\ref{fig:tri}. More precisely, we hope that the following computation is correct:

\begin{center}
$\sigma^\star((1 - \omega)^p - 1) = (1 - \omega)^p -1 + 1 = (1 - \omega)^p$
\end{center}
\bigskip

This computation is only valid, however, if the Eisenstein Mersenne prime in question is the correct associate. In order to verify if this is the case, we must analyze expressions of the form $(1 - \omega)^p$ to determine the values of $p$ that result in $M_p$ belonging to our region of usable associates. In what follows, we rewrite $(1 - \omega)^p$ as $a + b\omega$, where $a$ and $b$ are rational integers, and $h = \lfloor{\frac{p}{2}} \rfloor$. Due to the periodic behavior of $a$ and $b$ as $p$ increases, we need only consider the values of $p$ modulo 12, which are given in the top row of the table.

\begin{table}[h]
\setlength\extrarowheight{5pt}
\centering
\caption{Values of $a$ and $b$ such that $(1 - \omega)^p = a + b\omega$}
\label{my-label}
\begin{tabular}{|l|l|l|l|l|l|l|l|l|l|l|l|l|}
\hline
$p$ & 1 & 2 & 3 & 4 & 5 & 6 & 7 & 8 & 9 & 10 & 11 & 12 \\
\hline
$a$ & $3^h$ & 0 & $-3^h$ & $-3^h$ & $-2 \cdot 3^h$ & $-3^h$ & $-3^h$ & 0 & $3^h$ & $3^h$ & $2 \cdot 3^h$ & $3^h$ \\
\hline
$b$  & $-3^h$ & $-3^h$ & $-2 \cdot 3^h$ & $-3^h$ & $-3^h$ & 0 & $3^h$ & $3^h$ & $2 \cdot 3^h$ & $3^h$ & $3^h$ & 0 \\
\hline
\end{tabular}
\end{table}

It is clear that the only values of $p$ which result in the corresponding Eisenstein Mersenne prime being the appropriate associate are $p \equiv 11 \pmod{12}$.

Classifying perfect numbers in $\mathbb{Z}[\omega]$ using the norm--perfect approach suits our purposes well. Note that the structure of the even norm--perfect Eisenstein integer $\alpha$ in the next result parallels the characterization of even perfect numbers that first appeared in the works of \emph{Euclid}.

\begin{theorem}
Given any rational integer $k > 1$, if $(1-\omega)^k -1$ is an Eisenstein Mersenne prime and if $k \equiv 11 \pmod{12}$, then $\alpha = (1 - \omega)^{k - 1} [(1-\omega)^k -1]$ is an even norm--perfect Eisenstein integer.
\end{theorem}

\begin{proof}
To begin, we will simplify the expression $N[\sigma^\star(\alpha)]$ by using the multiplicative property of $\sigma^\star$:

\bigskip
\begin{align*}
N \left[ \sigma^\star(\alpha) \right] & = N \left[ \sigma^\star \left( \left( (1-\omega)^k -1 \right) \cdot (1-\omega)^{k-1} \right) \right]\\
& = N \left[ \sigma^\star \left( (1-\omega)^{k-1} \right) \cdot \sigma^\star \left( (1-\omega)^k -1 \right) \right]
\end{align*}
\bigskip

Next, we exploit the fact that the inputs to $\sigma^\star$ are a prime power and an Eisenstein Mersenne prime, respectively, together with the multiplicativity of the norm function.

\bigskip
\begin{align*}
N \left[ \sigma^\star \left( (1-\omega)^{k-1} \right) \cdot \sigma^\star \left( (1-\omega)^k -1 \right) \right] & = N \left[ \frac{(1-\omega)^k -1}{1-\omega -1} \cdot \left[ (1-\omega)^k -1 +1 \right] \right] \\
& = N \left[ \frac{(1-\omega)^k -1}{1-\omega -1} \right] \cdot N \left[ (1-\omega)^k -1 +1 \right]\\
& = N \left[ \frac{-1}{\omega} \cdot \left[ (1-\omega)^k -1 \right] \right] \cdot N \left[ (1-\omega)^k \right]\\
& = N \left[ \frac{-1}{\omega} \right] \cdot N \left[ (1-\omega)^k -1 \right] \cdot N \left[ (1-\omega)^k \right]
\end{align*}
\bigskip

Our desired result now easily follows from the fact that $- \omega^2$ is a unit, together with properties of the norm function:

\begin{align*}
N \left[ \frac{-1}{\omega} \right] \cdot N \left[ (1-\omega)^k -1 \right] \cdot N \left[ (1-\omega)^k \right] & = N\left [-\omega ^2 \right] \cdot N \left[ [(1-\omega)^k -1] \cdot (1-\omega)^k \right]\\
& = N \left[ \left( (1-\omega)^k -1 \right) \cdot (1-\omega)^{k-1} (1 - \omega) \right]\\
& = N \left[ (1-\omega) \right] \cdot N \left[ \left( (1-\omega)^k -1 \right) \cdot (1-\omega)^{k-1} \right]\\
& = N \left[ 1-\omega \right] \cdot N \left[ \alpha \right]\\
%& = \left( 1^2 -(1)(-1)+(-1)^2 \right) \cdot N \left[ \alpha \right]\\
& = 3 \cdot N \left[ \alpha \right].
\end{align*}

\end{proof}

%section5
\section{Future work}\label{sec:work}

We have been unable to show that every even norm-perfect Eisenstein integer is of the form given in Theorem 4.3, but we suspect that this is indeed the case, in part because an extensive computer search using Mathematica has to date failed to find any other types of even norm-perfect Eisenstein integers.

Similarly, we have failed to find even a single odd norm-perfect Eisenstein integer. But if such an object exists, we believe it satisfies the following.

\begin{conjecture}
Any odd norm--perfect Eisenstein integer $\alpha$ must be of the form $\alpha = \pi^k \gamma^3$, where $\pi$ and $\gamma$ are both odd Eisenstein integers, $\pi$ is an odd prime, $k\equiv 2 \pmod{3}$, and $\pi$ and $\gamma$ share no common non-unit factors.
\end{conjecture}

\bibliographystyle{amsplain}
\bibliography{Eisenstein}

\end{document}